\documentclass[11pt]{article}
\usepackage[margin=1in]{geometry}
\usepackage{amsmath,latexsym,amssymb,amsthm,hyperref}
\usepackage{comment}
\usepackage{mathtools}
\usepackage{enumerate}
\usepackage[pagewise]{lineno}
\usepackage{thmtools}
\usepackage{thm-restate}

\newtheorem{thm}{Theorem}[section]
\newtheorem{cor}[thm]{Corollary}

\newtheorem{conj}[thm]{Conjecture}
\theoremstyle{definition}
\newtheorem{defn}[thm]{Definition}
\newtheorem{ques}[thm]{Question}
\newtheorem*{remark}{Remark}
\newtheorem*{claim}{Claim}

\newcommand{\Qn}{Q^n}

\DeclarePairedDelimiter{\ceil}{\lceil}{\rceil}
\title{Stability for hyperplane covers}
\author{Shagnik Das\thanks{Department of Mathematics, National Taiwan University, Taiwan. E-mail: \texttt{shagnik@ntu.edu.tw}.} \and Valjakas Djaljapayan\thanks{Department of Mathematics, National Taiwan Normal University, Taiwan. E-mail:\texttt{81040002s@ntnu.edu.tw}} \and Yen-Chi Roger Lin\thanks{Department of Mathematics, National Taiwan Normal University, Taiwan. E-mail: \texttt{yclin@math.ntnu.edu.tw}} \and Wei-Hsuan Yu\thanks{Department of Mathematics, National Central University, Taiwan. E-mail: \texttt{whyu@math.ncu.edu.tw}}}

\makeatletter
\newcommand{\subjclass}[2][1991]{%
  \let\@oldtitle\@title%
  \gdef\@title{\@oldtitle\footnotetext{#1 \emph{Mathematics subject classification.} #2}}%
}
\newcommand{\keywords}[1]{%
  \let\@@oldtitle\@title%
  \gdef\@title{\@@oldtitle\footnotetext{\emph{Key words and phrases.} #1.}}%
}
\makeatother



\begin{document}
\maketitle

\begin{abstract}
  An almost $k$-cover of the hypercube $\Qn = \{0,1\}^n$ is a collection of hyperplanes that avoids the origin and covers every other vertex at least $k$ times. When $k$ is large with respect to the dimension $n$, Clifton and Huang asymptotically determined the minimum possible size of an almost $k$-cover. Central to their proof was an extension of the LYM inequality, concerning a weighted count of hyperplanes.

  In this paper we completely characterise the hyperplanes of maximum weight, showing that there are $\binom{2n-1}{n}$ such planes. We further provide stability, bounding the weight of all hyperplanes that are not of maximum weight. These results allow us to effectively shrink the search space when using integer linear programming to construct small covers, and as a result we are able to determine the exact minimum size of an almost $k$-cover of $Q^6$ for most values of $k$. We further use the stability result to improve the Clifton--Huang lower bound for infinitely many choices of $k$ in every sufficiently large dimension $n$.
\end{abstract}

\section{Introduction}
\label{sec:intro}

While it is clear that the hypercube $\Qn = \{0,1\}^n$ can be covered by two hyperplanes, a classic and surprising result shows that if we have to avoid the origin, then $n$ planes are needed to cover the remaining points. Lying at the intersection of finite geometry and extremal combinatorics, this problem and its variations have been studied by several researchers over the decades. In the finite geometry setting, such a hyperplane cover is closely related to the notion of blocking sets, and research in this direction was pioneered in the late 1970s by Jamison~\cite{Jam77}. Meanwhile in the extremal context, this problem was first studied by Alon and F\"uredi~\cite{alon1993covering}, who resolved a problem of Komj\'ath~\cite{komjath1994partitions} from Ramsey Theory. The Alon--F\"uredi Theorem was a precursor to the hugely influential Combinatorial Nullstellensatz, and indeed, this problem has proven a valuable test case in the development of the algebraic method. For a more thorough survey of the history of this problem, we refer the reader to~\cite{BBDM23}.

\subsection{Covering with multiplicities}

In recent years, renewed interest in this hyperplane covering problem was sparked by the work of Clifton and Huang~\cite{clifton2020almost}. In this paper, they studied the multiplicity version of the problem. Given $k \in \mathbb{N}$, we define an \emph{almost $k$-cover} of $\Qn$ to be a set of hyperplanes that avoids the origin while covering all other points of $\Qn$ at least $k$ times. We are then interested in the minimum size of an almost $k$-cover, a quantity we denote by $f(n,k)$. Note that the Alon--F\"uredi Theorem shows $f(n,1) = n$.

Clifton and Huang showed that the extremal function $f(n,k)$ exhibits different behaviour, depending on the relative sizes of the two parameters. When $n$ is large compared to $k$, they used algebraic methods to obtain lower bounds on $f(n,k)$, showing that $f(n,2) = n+1$, $f(n,3) = n+3$ for all $n \ge 2$, and $f(n,k) \ge n + k + 1$ whenever $k \ge 4$ and $n \ge 3$. In a subsequent paper, Sauermann and Wigderson~\cite{SW22} solved the algebraic version of this problem, which in particular improves the lower bound to $f(n,k) \ge  n + 2k - 3$ for any $k \ge 2$ and $n \ge 2k-3$. However, these lower bounds fall short of Clifton and Huang's upper bound of $f(n,k) \le n + \binom{k}{2}$, which they conjecture to be the truth for all $k \in \mathbb{N}$ and $n$ sufficiently large with respect to $k$.

In this paper, however, we will be interested in the other regime, where $k$ is large with respect to $n$. In this range, Clifton and Huang~\cite{clifton2020almost} determined asymptotically the size of the smallest almost $k$-covers, showing $f(n,k) = \left(H_n + o(1) \right)k$, where $H_n = 1 + \frac12 + \frac13 + \cdots + \frac{1}{n}$ is the $n$th Harmonic number. To prove this result, they solved the linear programming relaxation of the integer linear program that represents the hyperplane covering problem.
In particular, to prove the lower bound, they first defined a weighting of the points of $\Qn$.

\begin{defn}
  Let $x$ be a nonzero point in $\Qn$ with exactly $t$ coordinates equal to $1$. The \emph{weight} of $x$ is defined as
  \[
    w(x) = \frac{1}{t {\binom{n}{t}}}.
  \]	
  More generally, given any set $S \subseteq \mathbb{R}^n$, we define the \emph{weight} of $S$ to be the sum of the weights of all points in $S \cap \Qn$; that is,
  \begin{equation}
  \label{eq:weight-at-most-1}
  w(S) = \sum_{p \in S \cap \Qn} w(p). \end{equation}
\end{defn}

Given this definition, the crucial step was following theorem.

\begin{thm}[\cite{clifton2020almost}, Theorem 1.3]
  \label{thm:CH}
  If $h$ is a hyperplane avoiding the origin, then $w(h) \le 1$.
\end{thm}

To see how this implies the lower bound, observe that the total weight of the hypercube is $H_n$. Since every point is covered at least $k$ times, it follows that the total weight of the hyperplanes in an almost $k$-cover must be at least $H_n k$. Since, by Theorem~\ref{thm:CH}, each hyperplane can have weight at most $1$, we must have 
\begin{equation} \label{eqn:LPlowerbound}
f(n,k) \ge \ceil*{ H_n k }. 
\end{equation}

There are a couple of key remarks to be made concerning this theorem. First, Clifton and Huang observed that the bound is tight; any hyperplane $h = \{ x: \sum_{i=1}^n c_i x_i = 1 \}$ where all the coefficients $c_i$ are all either $0$ or $1$ satisfies $w(h) = 1$. In particular, this gives exponentially many hyperplanes of maximum weight.

Next, Clifton and Huang observed that if the coefficients $c_i$ are all positive, then the points covered by the hyperplane are the characteristic vectors of the sets in an antichain. The bound then follows immediately from the famous Lubell--Yamamoto--Meshalkin inequality~\cite{bollobas1965generalized, lubell1987short, meshalkin1963generalization, yamamoto1954logarithmic}. Thus, Theorem~\ref{thm:CH} can be viewed as a generalisation of the LYM inequality.

\subsection{Our results}

In this paper, we prove stability for Theorem~\ref{thm:CH}, characterising all hyperplanes of maximum weight, and improving the bound on the weight of all other hyperplanes.

\begin{thm}
  \label{thm:main}
  Let $h = \{ x : \sum_{i=1}^n c_i x_i = 1 \}$ be an affine hyperplane in $\mathbb{R}^n$ that does not pass through the origin. If the coefficients satisfy
  \begin{enumerate}[(i)]
    \item $c_i \leq 1$ for all $i$,
    \item $\sum_{i=1}^n c_i \geq 1$, and
    \item $c_i \in \mathbb{Z}$ for all $i$,
  \end{enumerate}
  then $w(h) = 1$. Otherwise, $w(h) \le 1 - \frac1n$.
\end{thm}

Note that the upper bound for the non-maximum-weight hyperplanes is best possible. Indeed, the hyperplane $h_{\alpha} = \{ x : \alpha x_1 + \sum_{i=2}^n x_i = 1 \}$ has $w(h_{\alpha}) = 1 - \frac1n$ whenever $\alpha \notin \{-(n-2), -(n-3), \hdots, -1, 0, 1\}$.

Our characterisation allows us to enumerate the hyperplanes of maximum weight.

\begin{cor} \label{cor:countweight1}
  The number of weight-$1$ hyperplanes in $\mathbb{R}^n$ is $\binom{2n-1}{n}$.
\end{cor}

\begin{proof}
 We enumerate the weight-$1$ hyperplanes by establishing a bijection from the set $B_n \subseteq \{-1, +1\}^{2n-1}$ of all sequences $(b_1, \dots, b_{2n-1})$ with exactly $n$ positive entries.  
  For a sequence $(b_1, \dots, b_{2n-1})$ in $B_n$, let $i_\ell$ be the indices such that $b_{i_\ell} = +1$ for $\ell = 1, 2, \dots, n$; we arrange these indices so that $i_1 < i_2 < \cdots < i_n$, and we set $i_0 = 0$ for later usage.  Now we define a mapping $\varphi$ from $B_n$ to $\mathbb{Z}^n$ as follows: for $b = (b_1, \dots, b_{2n-1}) \in B_n$, let $\varphi(b) = (c_1, \hdots, c_n)$, where the coefficients $c_\ell$ are given by
  \[
    c_\ell = \sum_{j = i_{\ell-1}+1}^{i_\ell} b_j, \qquad \ell = 1, 2, \dots, n.
  \]
  For instance, $\varphi(+1,+1,-1,-1,+1,-1,+1,+1,-1) = (1, 1, -1, 0, 1)$. 
  
  It is clear that the image of $\varphi$ is exactly those vectors $(c_1, \hdots, c_n)$ that satisfy the conditions (i), (ii), and (iii) in Theorem~\ref{thm:main}, and the inverse mapping $\varphi^{-1}$ can be easily defined: simply expand every coefficient $c_\ell$ into a sequence of $|c_\ell|+1$ terms of $-1$ followed by a $+1$. This shows that $\varphi$ is a bijection, and the number of weight-1 hyperplanes in $\mathbb{R}^n$ is the cardinality of $B_n$.
\end{proof}

The characterisation of maximum weight hyperplanes is also useful for determining precise values of the extremal function $f(n,k)$. As an example of an application of our result, we will prove the following, extending results of Clifton and Huang (who determined $f(5,k)$ for all $k \ge 15$).

\begin{restatable}{thm}{nequalssix} \label{thm:nequals6}
For $k \ge 65$, we have
\[ f(6,k) = \ceil*{\frac{49k}{20}}\]
whenever $k \not\equiv 2, 11 \pmod{20}$.
\end{restatable}

\paragraph{Organisation}

In Section~\ref{sec:GBR}, we will review Clifton and Huang's proof of Theorem~\ref{thm:CH}, introducing terminology that we will need in our own proofs. In Section~\ref{sec:proof}, we use this framework to prove Theorem~\ref{thm:main}, and prove some further stability. In Section~\ref{sec:covers}, we combine our stability result with the linear programming approach to prove Theorem~\ref{thm:nequals6}, and also show that the lower bound~\eqref{eqn:LPlowerbound} is not tight infinitely often. Finally, in Section~\ref{sec:discussion}, we provide some concluding remarks and outline possible directions for further research.

\section{The good, bad, and the redundant} \label{sec:GBR}

In this section we will review the proof of Theorem~\ref{thm:CH}, primarily with the aim of establishing some terminology and notation that we shall use in our own proofs in the next section.

Let $h$ be the affine hyperplane in $\mathbb{R}^{n}$ defined by the equation $\sum_{i=1}^n c_i x_i = 1$. We wish to show that $w(h) \le 1$; that is, the sum of the weights of the points of $\Qn$ contained in $h$ is at most $1$. For this, it is useful to identify a point $x \in \Qn$ with the subset $S \subseteq [n]$ whose characteristic vector it is; that is, $\{ i : x_i = 1\}$. The idea behind the proof is to associate each set $S$ covered by the hyperplane $h$ with a disjoint set of permutations in $S_n$, whose size is proportional to the weight of $x$. Since the total number of permutations is bounded by $n!$, this will in turn bound the weight of the hyperplane $h$.

\begin{defn}
Given some set $S\subseteq [n]$ such that $\sum_{i\in S}c_{i}=1$, we say a permutation $\pi\in S_n$ \emph{yields} $S$ provided that 
\begin{enumerate}[(i)]
\item $\pi([|S|])=S;$ that is, $S$ is an initial segment of $\pi$, and

\item $\displaystyle\sum_{i=1}^{\ell}c_{\pi(i)}\begin{cases}
<1 &\text{ if }\ell<|S|,\\
=1 & \text{ if }\ell=|S|,\\
\end{cases} $ where $1 \le \ell \le |S|$.  
\end{enumerate}
\end{defn}

From the definition, it is clear that each $\pi \in S_{n}$ can yield at most one subset $S\subseteq [n]$. We call a permutation $\pi \in S_{n}$ \emph{bad} if it does not yield any set $S\subseteq [n]$, and define $\mathcal{B}$ to be the collection of all bad permutations in $S_n$.

Clifton and Huang~\cite{clifton2020almost} proved that, for every set $S \subseteq [n]$ such that $\sum_{i\in S} c_i = 1$, and for every cyclic permutation $\sigma$ of $S$, there is at least one starting point in $S$ such that if we unfold $\sigma$ to a linear permutation $\pi_{\sigma}$ of $S$, and extend it arbitrarily to any $\pi\in S_n$ with $\pi|_{[|S|]}=\pi_{\sigma}$,  then $\pi$ yields $S$. In the case where $\sigma$ admits two or more such starting points, we call $\sigma$ \emph{switchable}, and from the available options, we choose $\pi_\sigma$ such that the initial entry $\pi_{\sigma}(1)$ is the largest. 

Now, consider a permutation $\pi\in S_n$. If $\pi$ is not bad, there is some set $S$ that $\pi$ yields. Let $\sigma$ be the cyclic permutation of $S$ induced by $\pi |_{[|S|]}$. If $\pi |_{[|S|]}= \pi_{\sigma}$, we say that $\pi$ is \emph{good}, and otherwise we say that $\pi$ is \emph{redundant}. We let $\mathcal{G} \subseteq S_n$ be the set of good permutations, and $\mathcal{R} \subseteq S_n$ the set of redundant permutations. Note that this gives a partition $S_n = \mathcal{G} \cup \mathcal{B} \cup \mathcal{R}$ of the set of the permutations of $[n]$ into the subsets of good, bad, and redundant permutations, whence
\begin{equation}
\label{eq:GBR}
n! = |\mathcal{G}| + |\mathcal{B}| + |\mathcal{R}|.
\end{equation}

Given a set $S$ covered by $h$, note that there are $(|S|-1)!$ cyclic permutations $\sigma$ of $S$, each of which gives $(n-|S|)!$ good permutations (which must start with $\pi_{\sigma}$). Hence, abusing notation to write $S \in h$ to indicate that $S$ is covered by $h$,
\begin{equation*}
|\mathcal{G}|=\sum_{S\in h} (|S|-1)!(n-|S|)!,
\end{equation*}
or
\begin{equation*}
\frac{|\mathcal{G}|}{n!}=\sum_{S\in h}\frac{(|S|-1)!(n-|S|)!}{n!}=\sum_{S\in h}\frac{1}{|S|\binom{n}{|S|}} = \sum_{x \in h} w(x) = w(h).
\end{equation*}
Dividing~\eqref{eq:GBR} by $n!$ and rearranging then yields 
\begin{equation}
w(h)=\frac{|\mathcal{G}|}{n!}=1-\frac{|\mathcal{B}|}{n!}-\frac{|\mathcal{R}|}{n!}\leq 1\label{eq:wh}.
\end{equation}

This proves Theorem~\ref{thm:CH}. Moreover, it outlines how one can establish stability --- to prove that non-maximum-weight hyperplanes have weight bounded away from $1$, we need to show that they admit many bad or redundant permutations.

\section{Weights of hyperplanes}
\label{sec:proof}

We will now use the framework set up in the previous section to prove our main result, Theorem~\ref{thm:main}. Following that, we shall prove some further stability, providing a partial characterisation of all hyperplanes of large weight.

\subsection{The proof of Theorem~\ref{thm:main}}

We first prove that if $h = \{ x : \sum_{i=1}^n c_i x_i = 1 \}$, where the coefficients $c_i$ satisfy the three conditions in Theorem~\ref{thm:main}, then $w(h) = 1$. Per the discussion in Section~\ref{sec:GBR}, this amounts to showing that every permutation in $S_n$ is good, i.e., $\mathcal{G} = S_n$.

For any permutation $\pi \in S_n$, let $t$ be the smallest index such that
\begin{equation*}
 \quad \sum_{j=1}^t c_{\pi(j)} = 1.
\end{equation*}
Note that the conditions (i), (ii) and (iii) guarantee that such an index exists, and moreover that $\sum_{j=1}^{\ell} c_{\pi(j)} < 1$ for all $1 \le \ell \le t-1$. Hence, $\pi$ yields the subset $S = \{ \pi(1), \dots, \pi(t) \}$.

  Now consider the cyclic permutation $\sigma = \langle \pi(1), \dots, \pi(t) \rangle$ of $S$.  We argue that $\pi(1)$ is the unique starting point for $\sigma$, that is, $\pi|_{[t]} = \pi_{\sigma}$.
  There is nothing to prove if $t=1$, so we may assume $t>1$ and $\pi(s)$ is another starting point of $\sigma$ with $1 < s \leq t$.
  Then $\sum_{i=s}^t c_{\pi(i)} < 1$.  Since all the coefficients $c_j$'s are integers, this implies that $\sum_{i=s}^t c_{\pi(i)} \leq 0$.  But this leads to
  \[
    \sum_{i=1}^{s-1} c_{\pi(i)} = \sum_{i=1}^t c_{\pi(i)} - \sum_{i=s}^t c_{\pi(i)} \geq 1 - 0 = 1,
  \]
  which contradicts our choice of $t$.  Hence we have $\pi|_{[t]} = \pi_{\sigma}$, and so $\pi$ is good. As $\pi$ was an arbitrary permutation, it follows that $\mathcal{G} = S_n$, as required.

  \medskip

  To complete the proof of Theorem~\ref{thm:main}, we need to show that if $w(h) < 1$, then $w(h) \le 1 - \frac1n$. Note that by what we have just shown, we know at least one of the following three cases must occur:
\begin{enumerate}[(1)]
    \item $c_{i}>1$ for some $i\in [n]$,

    \item $\sum_{i=1}^{n}c_{i}<1$, or 

    \item $c_{i}\notin\mathbb{Z}$ for some $i\in [n]$.
\end{enumerate}

We treat each of these cases in turn.

\paragraph{Case (1)} $c_{i}>1$ for some $i\in [n]$. 

Observe that any permutation $\pi$ with $\pi(1)=i$ is bad, since if a permutation yields a set $S\in [n]$, its first $|S|$ partial sums are at most $1$. Thus, $|\mathcal{B}|\geq (n-1)!$, and so by \eqref{eq:wh} we have $w(h)\leq 1-\frac{1}{n}$.

\paragraph{Case (2)} $\sum_{i=1}^{n} c_{i} < 1$.

We will show that each cyclic permutation $\sigma$ of $[n]$ admits a starting point for which the resulting permutation $\pi\in S_n$ is bad. Then $|\mathcal{B}|\geq (n-1)!$, which once again implies $w(h)\leq 1-\frac{1}{n}$. 

Now let $\pi'$ be the permutation we obtain from $\sigma$ by starting at $1$.  For $1\leq t\leq n$, define $a_{t} =\sum_{i=1}^{t}c_{\pi^{\prime}(i)}$. Let $t_{0} \in [n]$ be any index such that $a_{t_{0}}\geq a_{t}$ for all $t\in [n]$. In particular, 
\begin{equation*}
a_{t_{0}}\geq a_{n}=\sum_{i=1}^{n}c_{\pi'(i)}=\sum_{i=1}^{n}c_{i}.
\end{equation*}

If $t_0 = n$, then $a_t \leq a_{t_0} = a_n < 1$ for each $t \in [n]$, which means that $\pi'$ is already bad.
Otherwise, take $\pi$ to be the permutation obtained by starting $\sigma$ at $t_{0}+1$ . We shall show that $\pi$ is bad. If not, then it yields some set $S \subseteq [n]$, and so there is some $j$ with $\sum_{i=1}^{j}c_{\pi(i)}=1$. If $j\leq n-t_{0}$ then 
\begin{equation*}
\sum_{i=1}^{j}c_{\pi(i)}=\sum_{i=t_{0}+1}^{t_{0}+j}c_{\pi^{\prime}(i)}=a_{t_{0}+j}-a_{t_{0}}\leq 0,
\end{equation*}  
by the choice of $t_{0}.$ Otherwise, if $n-t_{0}+1\leq j\leq n,$ then 
\begin{equation*}
\sum_{i=1}^{j}c_{\pi(i)}=\sum_{i=t_{0}+1}^{n}c_{\pi^{\prime}(i)}+\sum_{i=1}^{j-n+t_{0}}c_{\pi^{\prime}(i)}=a_{n}-a_{t_{0}}+a_{j-n+t_{0}}\leq a_{n}<1
\end{equation*}
where the last inequality holds by $a_{t_{0}}\geq a_{j-n+t_{0}}$ and the strict inequality follows by assumption. Hence, there is no such $j,$ and $\pi$ must be bad, completing this case.

\paragraph{Case (3)} $c_{i}\notin\mathbb{Z}$ for some $i\in [n]$. 

If $c_{i}$ is the unique non-integral coefficient, then any permutation with $\pi(1)=i$ is bad, as all its partial sums will be non-integral. Thus, as in Case (1), we will have $w(h)\leq 1-\frac{1}{n}$. Hence we may assume there are $r\geq 2$ non-integral coefficients, which we may assume (without loss of generality) to be $c_{1},c_{2},\dots,c_{r}$. In this case, we shall show that for every cyclic permutation of $[n]$ there is a starting point for which the corresponding permutation $\pi\in S_{n}$ is either bad or redundant. Then we have $|\mathcal{B}| + |\mathcal{R}|\geq (n-1)!$, and so it again follows from $(\ref{eq:wh})$ that $w(h)\leq 1-\frac{1}{n}.$ 

Let $\sigma$ be a cyclic permutation of $[n]$, and partition it into the cyclic intervals $I_{1},I_{2}, \dots,I_{r}$ starting $1,2,\dots,r$, respectively. If one of these intervals has a partial sum at least 1, then note that it is in fact strictly greater than 1; as $c_{i}$ is the only non-integral coefficient in $I_{i}$, all partial sums are non-integral. Thus, if we start $\pi$ at $i$, we obtain a bad permutation. Otherwise, for every interval $I_{i}$, all the partial sums are strictly less than 1. In this case, consider the permutation $\pi$ starting at $1$. If $\pi$ does not yield a set $S\subseteq [n],$ then it is bad. So, we may assume $\pi$ yields a set  $S\subseteq [n]$. Note that $S$ must contain some $c_j$ for $2\leq j\leq r$, since there must be another non-integral coefficient to make the sum integral. Hence, $S$ intersects more than one of the intervals $I_{i}$ in $\sigma$.

Now, for $1\leq t\leq |S|$, define $a_{t}=\sum_{i=t}^{|S|}c_{\pi (i)}$, and let $t_{0}$ be the smallest index that minimizes $a_{t}$. Note that we have $2\leq t_{0}\leq |S|$, since $\sum_{i=1}^{|S|}c_{\pi (i)}=\sum_{i\in S}c_{i}=1$, while if we choose $t$ be the index of the last non-integral coefficient in $S$ then $a_{t}$ is a partial sum of that interval, and hence $a_{t}<1$. Let $\pi'$ be the permutation obtained by swapping the intervals $\pi([1,t_{0}-1])$ and $\pi([t_{0},|S|])$. Note that we still have $\pi'([|S|])=S$, and that $\pi$ and $\pi'$ induce the same cyclic permutation of $S$.
\begin{claim}
$\pi'$ also yields $S$.
\end{claim}
\begin{proof}[Proof of Claim.]
 Since $\pi^{\prime}([|S|])=S$, the only way that $\pi'$ could fail to yield $S$ is if there is some index $j \in [|S|-1]$ such that $\sum_{i=1}^{j} c_{\pi'(i)}\geq 1$. If $j\leq |S|-t_{0}+1$ then $\pi'([j])=\pi([t_{0},t_{0}+j-1]).$ If this partial sum is positive, let alone greater than or equal to $1$, then 
\begin{align*}
a_{t_{0}+j}=\sum_{i=t_{0}+j}^{|S|}c_{\pi(i)}=\sum_{i=t_{0}}^{|S|}c_{\pi(i)}-\sum_{i=t_{0}}^{t_{0}+j-1}c_{\pi(i)}=a_{t_{0}}-\sum_{i=1}^{j}c_{\pi^{\prime}(i)}<a_{t_{0}},
\end{align*} 
which contradicts the choice of $t_{0}$ ($a_{t_0}$ is the minimum sum). Thus, $j\geq |S|-t_{0}+2$, in which case $\pi'([j])=\pi([t_{0}, |S|])\cup\pi([1,j-|S|+t_{0}-1])$ and \begin{equation*}
\sum_{i=1}^j c_{\pi'(i)} = \sum_{i=1}^{j-|S|+t_{0}-1}c_{\pi(i)}+\sum_{i=t_{0}}^{|S|}c_{\pi(i)}\geq 1.
\end{equation*}
Recall that $\pi$ yields $S$, and so we have $\sum_{i=1}^{|S|}c_{\pi(i)}=1$. We can therefore deduce that $\sum_{i=t_{0}+j-|S|}^{t_{0}-1}c_{\pi(i)}\leq 0$, and so
\begin{equation*}
a_{t_{0}+j-|S|}=\sum_{i=t_{0}+j-|S|}^{|S|}c_{\pi(i)}=\sum_{i=t_{0}+j-|S|}^{t_{0}-1}c_{\pi(i)}+\sum_{i=t_{0}}^{|S|}c_{\pi(i)}=\sum_{i=t_{0}+j-|S|}^{t_{0}-1}c_{\pi(i)}+a_{t_{0}}\leq a_{t_{0}},
\end{equation*}
which again contradicts our choice of $t_{0}$ (recall that $j\leq |S|-1$ which implies $t_{0}+j-|S|<t_{0}$). Hence, there is no such index $j$, and $\pi^{\prime}$ yields $S$ as well.
\end{proof}

Thus, we see that each of these cyclic permutations $\sigma$ of $S$ are switchable, admitting at least two starting points that result in permutations yielding $S$. We defined $\pi_{\sigma}$ to be the permutation with the largest starting entry, and hence the permutation $\pi$ obtained from starting $\sigma$ at $1$ is redundant. Hence, we have shown that for each permutation of $[n]$, there is a starting point giving a bad or redundant permutation, which resolves Case (3).

\medskip

This completes the proof of Theorem~\ref{thm:main}.

\subsection{Further stability}

Theorem~\ref{thm:main} characterises the hyperplanes of maximum possible weight. It is then natural to ask what we can say about hyperplanes of large weight, and indeed, our methods allow us to at least partially describe their coefficients.

\begin{thm} \label{thm:furtherstab}
    Let $h = \{ x : \sum_{i=1}^n c_i x_i = 1 \}$ be a hyperplane in $\mathbb{R}^n$ of weight $w(h) > 1 - \frac{r}{n}$. Then the following hold:
    \begin{itemize}
        \item[(a)] $|\{i : c_i > 1 \}| \le r-1$, and
        \item[(b)] $|\{ i : c_i \notin \mathbb{Z} \}| \le 2r^2 - 1$.
    \end{itemize}
\end{thm}

Before we embark on the proof, we first note that the bound in part (a) is best possible. Indeed, if $c_1 = \hdots = c_{r-1} = 2, c_{r} = \hdots = c_n = 1$, then it is easy to see that a permutation $\pi$ is bad if $\pi(1) \in [r-1]$, and good otherwise. This implies $w(h) = 1 - \frac{r-1}{n} > 1 - \frac{r}{n}$, and we have $r-1$ coefficients larger than $1$.

On the other hand, we do not expect the bound in part (b) to be tight, and rather expect that there can be at most $O(r)$ non-integral coefficients. However, note that a bound of $r$ is not true in general. For instance, if $n = 2r-1$, then we can take $c_i = \frac12$ for all $i$; this has weight $\frac12$ but $n = 2r-1$ fractional coefficients. Perhaps, though, if one imposes some lower bound on $n$, then it could be true that one must have fewer than $r$ fractional coefficients.

Finally, we note that Theorem~\ref{thm:furtherstab} only has two parts in its characterisation, as compared to the three parts in Theorem~\ref{thm:main}. One might also expect that, if the hyperplane has weight close to one, then the sum of the coefficients should not be too small. However, this is not true --- one could have $c_1 = c_2 = \hdots = c_{n-1} = 1$, with $c_n$ tending to $- \infty$. Then the sum of the coefficients is very small, but all permutations except those starting with $n$ are good, meaning $w(h) = 1 - \frac1n$. Hence, it appears to be somewhat complicated to formulate an appropriate condition on the sums of the coordinates.

\begin{proof} 
For (a), note that any permutation starting with a coefficient larger than $1$ is bad. Thus, if there are at least $r$ coefficients larger than $1$, we have $|\mathcal{B}| \ge r(n-1)!$, and then~\eqref{eq:wh} implies $w(h) \le 1 - \frac{r}{n}$.

\medskip

For (b), we make the following claim.

\begin{claim}
    If $h$ has $s$ fractional coefficients, then for any circular permutation of $[n]$ and $b \in \mathbb{N}$, there are either $b$ starting points that give bad permutations, or $\frac{s}{b}$ that give bad or switchable permutations.
\end{claim}

Let us first see how the claim gives the result. Suppose for contradiction that $h$ has $s = 2r^2$ fractional coefficients, and set $b = r$. Then we know that the $(n-1)!$ circular permutations of $[n]$ are of two types --- the first give rise to at least $r$ bad permutations, and the second give rise to $2r$ permutations that are either bad or switchable. Suppose there are $\alpha (n-1)!$ circular permutations of the first kind, and thus $(1 - \alpha)(n-1)!$ circular permutations of the second kind. If we let $\mathcal{S}$ denote the set of switchable permutations, we have $|\mathcal{B}| + |\mathcal{S}| \ge r\alpha (n-1)! + 2r(1 - \alpha)(n-1)!$, with $|\mathcal{B}| \ge r\alpha (n-1)!$.

Now notice that at least half of all switchable permutations are redundant, and so it follows from the above inequalities that $|\mathcal{B}| + |\mathcal{R}| \ge r(n-1)!$, and so the weight of the hyperplane is at most $1 - \frac{r}{n}$.
\end{proof}

To finish, we prove the claim.

\begin{proof}[Proof of Claim]

Fix a circular permutation $\sigma$ of $[n]$, and let $I_1, I_2, \hdots, I_s$ be the (cyclic) intervals that start with fractional coefficients, labelled in cyclic order. For $1 \le i \le s$, let $\pi_i$ be the linear permutation of $[n]$ obtained from $\sigma$ by starting at the $i$th fractional coefficient (so $I_i$ is an initial segment of $\pi_i$).

Now since each $I_i$ starts with a fractional coefficient, and has no other, all of its initial sums are fractional. Thus, if $I_i$ has any initial sum that is at least $1$, it is strictly larger than $1$, and so $\pi_i$ will be a bad permutation. Thus, if there are at least $b$ such intervals, then we obtain $b$ bad permutations, and we are done.

Suppose instead that $I_i$ has an initial sum in the interval $(0,1)$. If the permutation $\pi_i$ is good, then it yields a set $S$. Thus, if we consider the tail interval of $S$ with the smallest sum, this will be a proper subinterval of $S$ (since all of $S$ has sum $1$, but we can drop the initial sum in $I_i$ with positive sum to obtain a tail interval with smaller sum). Then, as in the proof of Theorem~\ref{thm:main}, we can rotate $S$ to obtain another permutation that yields $S$ with the same cyclic permutation of $S$. Hence, $\pi_i$ is switchable in this case.

\medskip

In the remaining case, then, all initial sums of $I_i$, including $I_i$ itself, must be strictly negative. Suppose we have $b$ such intervals appearing consecutively in $\sigma$, and assume without loss of generality that these are $I_1, I_2, \hdots, I_b$. If all of these intervals are good, let $S_i$ be the set yielded by $\pi_i$.

First observe that $S_i$ and $S_{i+1}$ cannot end in the same interval. Indeed, if this were the case, then $S_i \triangle S_{i+1}$ would be $I_i \cup J$, where $J$ is the difference of two initial sums of the interval in which $S_i$ and $S_{i+1}$ end. However, the sum of $I_i$ is fractional, while the sum of $J$ must be an integer, and thus we cannot have both $S_i$ and $S_{i+1}$ having sum equal to $1$.

Next, observe that $S_i$ must end after $S_{i+1}$. If not, then $S_{i+1}$ contains $S_i \setminus I_i$ as an initial sum. However, the sum of $S_i$ is $1$, while the sum of $I_i$ is negative, so this means that $S_{i+1}$ has an initial sum that is strictly larger than $1$, which is a contradiction.

Hence, it follows that the sets $S_i$, $1 \le i \le b$, all end in distinct intervals. If $S_i$ ends in the interval $I_j$, then notice that $I_j \cap S_i$ is a tail sum of $S_i$, and it must be positive. If this sum is in $(0,1)$, then as before, we will be able to rotate $S_i$, which means that $\pi_i$ is switchable. Otherwise, $I_j$ contains an initial sum that is larger than $1$, but we only have at most $b-1$ such intervals. Thus, one of $\pi_1, \hdots, \pi_b$ must be switchable.

Hence, we have shown that every set of $b$ consecutive intervals has at least one bad or switchable permutation. By averaging over the $s$ intervals, it follows that there are at least $\frac{s}{b}$ bad or switchable permutations, as required.
\end{proof}

\section{Almost $k$-covers}
\label{sec:covers}

In this section, we show how our stability result can be used in the determination of $f(n,k)$, the minimum number of hyperplanes needed for an almost $k$-cover of the $n$-dimensional hypercube $\Qn$.

\subsection{Constructing small covers}

Recall that $f(n,k)$ is the solution to the integer linear program where we have a variable for every hyperplane, representing the multiplicity with which it appears in the cover, and a constraint for every nonzero point in $\{0,1\}^n$, ensuring the point is covered at least $k$ times.

Unfortunately, integer linear programming is notoriously difficult to solve, and this is especially true in this setting, as the number of variables involved grows incredibly quickly. Indeed, there are infinitely many hyperplanes in $\mathbb{R}^n$ that could, in principle, be included in our cover. However, we can finitise our problem by observing that we need only consider hyperplanes that intersect the hypercube \emph{maximally} --- that is, if $H_1$ and $H_2$ are hyperplanes, both avoiding the origin, and $H_1 \cap \{0,1\}^n \subseteq H_2 \cap \{0,1\}^n$, then we can replace any occurrence of $H_1$ in a cover with a copy of $H_2$.

Since any $n$ points in $\mathbb{R}^n$ determine a unique hyperplane, it follows that any maximally intersecting hyperplane must contain at least $n$ points of the hypercube. In particular, this implies that there are at most $\binom{2^n-1}{n}$ hyperplanes we need to consider, making the integer linear program finite.

Unfortunately, there is a considerable difference between \emph{finite} and \emph{computationally feasible}, and this upper bound of $\binom{2^n-1}{n}$ grows far too fast to leave us with any hope of computing exact answers even when $n$ is as small as $6$. While it should be pointed out that this is indeed just an upper bound, and some hyperplanes are significantly overcounted, brute-force enumeration of the maximally intersecting hyperplanes (which we could only carry out for $n \leq 5$) shows that the true number also exhibits rapid growth.

\begin{table}[ht]
    \centering
    \begin{tabular}{c|c|c|c|c|c|c}
        $n$ & 1 & 2 & 3 & 4 & 5 & 6 \\ \hline 
        Upper bound & 1 & 3 & 35 & 1365 & 169911 & 67945521 \\
        Actual count & 1 & 3 & 11 & 95 & 2629 & ?? \\
        Weight-$1$ planes & 1 & 3 & 10 & 35 & 126 & 462
    \end{tabular}
    \caption{The number of maximally intersecting and weight-$1$ hyperplanes in $\mathbb{R}^n$.}
    \label{tab:maxhyperplanes}
\end{table}

In order to be able to proceed computationally, then, it is necessary to restrict the search space. A natural place to start is by only considering hyperplanes of maximum weight. Indeed, the Clifton--Huang lower bound~\eqref{eqn:LPlowerbound} shows that in any almost $k$-cover, the total weight of the hyperplanes must be at least $H_n k$, where $H_n$ is the $n$th Harmonic number. To minimise the size of the cover, then, we would want each hyperplane to have as much weight as possible.

As we saw in Corollary~\ref{cor:countweight1}, there are far fewer hyperplanes of weight $1$; the number of these is the much more modest $\binom{2n-1}{n}$. Moreover, Theorem~\ref{thm:main} characterises these maximum-weight planes, so we are able to set up the corresponding integer linear program and efficiently search for small covers when $n = 6$. In many cases we are able to find covers matching the Clifton--Huang lower bound, as shown in Theorem~\ref{thm:nequals6}, which we first restate.

\nequalssix*

\begin{proof}
    The Clifton--Huang lower bound \eqref{eqn:LPlowerbound} gives $f(n,k) \geq \ceil*{H_6k} = \frac{49k}{20}$, and so we need only prove the upper bound. To this end, we solved the integer linear program corresponding to $f(n,k)$ for various values of $k$, restricting ourselves to using the $462$ hyperplanes of weight $1$.

    We noted that $f(6,60)=147$, which comes from the general construction given by Clifton and Huang~\cite{clifton2020almost}. We next considered the case $k = 20$. The lower bound implies $f(6,20) \ge 49$, with equality only possible if all hyperplanes involved have weight $1$. Our integer linear program solver found such an almost-$k$ cover, which we have provided in Appendix~\ref{app:20cover}. Note that this cover includes hyperplanes with coefficients other than $0$ or $1$; that is, it uses some of the new weight-$1$ hyperplanes described in Theorem~\ref{thm:main}.

    Having thus established that $f(6,20) = 49$, we turn to other values of $k$. Now, since $f(n,k + \ell) \leq f(n,k) + f(n,\ell)$, these together imply that $f(6,k + 20m) \leq f(6,k) + 49m$ for all integers $m \geq 1$. We then solved the integer linear program for all $23 \leq k \leq 42$, finding that $f(n,k) = \ceil*{\frac{49k}{20}}$ in all cases except $k \in \{31, 33, 42\}$. While the first and last of these values are excluded from our result, we were able to resolve the case $k \equiv 13 \pmod{20}$ by finding $f(6,k) = \ceil*{\frac{49k}{20}}$ for $k = 53$.

    Using this finite set of computational results, together with the recursive upper bound, it follows that $f(6,k) \leq \ceil*{\frac{49k}{20}}$ for all $k \geq 65$ with $k \not\equiv 2, 11 \pmod{20}$.
\end{proof}

It is natural to ask what happens when $k \equiv 2, 11 \pmod{20}$. Our computational results were restricted to hyperplanes of weight $1$. However, by Theorem~\ref{thm:main}, we know that all other hyperplanes have weight at most $\frac{5}{6}$. Thus, if we have a cover of $m$ planes containing at least one plane that is not of weight $1$, the total weight of the planes in the cover is at most $m - \frac{1}{6}$.

For $k \equiv 2, 11 \pmod{20}$, we have $\frac{49k}{20} \geq \ceil*{\frac{49k}{20}} -  \frac{1}{10}$. Hence, in light of the previous remark, if there is a cover of size $\ceil*{ \frac{49k}{20}}$, it must consist entirely of weight-$1$ hyperplanes. Our solutions to the integer linear program show that for $k \in \{11, 22, 31, 42, 51, 62\}$, no such cover exists; in all these cases, we have $f(6,k) = \ceil*{\frac{49k}{20}} + 1$.

However, it is not inconceivable that further along these sequences, the Clifton--Huang lower bound is realised; that is, $f(6,2+20m_0) = 5 + 49m_0$ or $f(6,11+20m_0) = 27 + 49m_0$ for some suitably large $m_0$. If this does happen, then using $f(6,20) = 49$, it again follows that we would have equality for all $m \geq m_0$.

Indeed, we also have $\frac{49k}{20} > \ceil*{\frac{49k}{20}} - \frac16$ when $k \equiv 13 \pmod{20}$, which means that any cover of size $\ceil*{\frac{49k}{20}}$ can only contain weight-$1$ hyperplanes. Our solutions to the integer linear program show that there are no such covers for $k = 13$ or $k = 33$, but it does exist for $k = 53$ (and, therefore, for all further terms in this sequence).

\subsection{Improving the lower bound}

In some cases, though, one can show that the lower bound $f(n,k) \geq \ceil*{H_n k}$ is not tight for arbitrarily large $k$. Using numerical data about the hyperplanes in $\mathbb{R}^5$, Clifton and Huang were able to show that $f(5,67+60m) = \ceil*{\frac{137k}{60}} + 1$ for all $m \ge 1$. Using our stability result, we can extend this to all sufficiently large $n$, showing that the lower bound is not tight infinitely often.

\begin{thm} \label{thm:plusone}
Let $k \geq 2$ and $n \in \mathbb{N}$ be such that $0 < \ceil*{kH_n} - kH_n < \frac{1}{n\binom{n-1}{\lfloor n/2 \rfloor}}$. Then $f(n,k) \geq \ceil*{kH_n} + 1$.

In particular, if $n$ is sufficiently large, there are infinitely many choices of $k$ for which $f(n,k) \geq \ceil*{k H_n} + 1$.
\end{thm}

\begin{proof}
Consider a smallest almost $k$-cover $\mathcal{H}$ of $\Qn$, and for each hyperplane $h$, let $x_h$ denote the number of copies of $h$ in $\mathcal{H}$. We then have
\[ \left| \mathcal{H} \right| = \sum_h x_h = \sum_h \left( \sum_{S: S 
 \in h} w(S) + 1 - \sum_{S: S \in h} w(S) \right) x_h. \]

We obtain
\[ \left| \mathcal{H} \right| = \sum_S \left( \sum_{h: h \ni S} x_h \right) w(S) + \sum_h \left( 1 - \sum_{S: S \in h} w(S) \right) x_h. \]

Now, since $\mathcal{H}$ is an almost $k$-cover, we have $\sum_{h : h \ni S} x_h \geq k$ for all $S$. Since
\[ \sum_S w(S) = \sum_{s = 1}^n \sum_{S: |S| = s} w(S) = \sum_{s=1}^n \frac{\binom{n}{s}}{s \binom{n}{s}} = \sum_{s=1}^n \frac{1}{s} = H_n, \]
we can further rearrange this equation to obtain
\begin{equation}
    \label{eq:H-KHn}
 \left| \mathcal{H} \right| - k H_n = \sum_S \left( \sum_{h : h \ni S} x_h - k \right) w(S) + \sum_h \left( 1 - \sum_{S: S \in h} w(S) \right) x_h. \end{equation}

Consider the terms on the right-hand side of (\ref{eq:H-KHn}). Since $\mathcal{H}$ is an almost $k$-cover, every set is covered at least $k$ times, so $\sum_{h : h \ni S} x_h -k$ is always non-negative integer. Furthermore, by Theorem~\ref{thm:main}, each hyperplane has weight either exactly $1$ or at most $1 - \frac{1}{n}$, which means $1 - \sum_{S:S \in h} w(S)$ is non-negative, and is at least $\frac{1}{n}$ when it is positive. Thus, each individual summand on the right-hand side is non-negative, and the positive summands are at least $\min \left( \min_S w(S), \frac{1}{n} \right) = \frac{1}{n \binom{n-1}{\lfloor n/2 \rfloor }}$.

Now we look at the left-hand side. Since it must be non-negative, we have $\left| \mathcal{H} \right| \geq \ceil*{H_n k}$. Suppose we had equality. By our assumption on $k$ and $n$, $0 < \ceil*{kH_n} - kH_n < \frac{1}{n \binom{n-1}{\lfloor n/2 \rfloor}}$, but we previously established that the right-hand side, if positive, must be at least $\frac{1}{n \binom{n-1}{\lfloor n/2 \rfloor}}$. Hence, we must have $f(n,k) = \left| \mathcal{H} \right| \geq \ceil*{k H_n} + 1$.

\medskip

For the second assertion, let $H_n = \frac{c_n}{d_n}$, where $c_n$ and $d_n$ are coprime. We can then find some $1 \le k_0 \le d_n$ such that $k c_n\equiv -1 \pmod{d_n}$. It then follows that for any $k \equiv k_0 \pmod{d_n}$, we have $\ceil*{kH_n} - kH_n = \frac{1}{d_n}$.

As shown by Boyd~\cite{Boy94}, $d_n = e^{(1+ o(1))n}$. On the other hand, $n \binom{n-1}{\lfloor n/2 \rfloor} \leq n2^{n-1}$. Hence, for $n$ sufficiently large, we have $\ceil*{H_n k} - H_n k < \frac{1}{n \binom{n-1}{\lfloor n/2 \rfloor}}$. Thus, for this infinite sequence of values of $k$, $f(n,k) \ge \ceil*{ H_n k} + 1$.
\end{proof}

\begin{remark}
Note that in the proof of the second statement, it suffices to have $k \equiv k_0 \pmod{d_n}$ where $k_0$ is such that $c k_0 \equiv -r \pmod{d_n}$ for some $1 \le r < \frac{d_n}{n \binom{n-1}{\lfloor n/2 \rfloor}}$. This shows that the linear programming lower bound is not tight for approximately $(e/2)^{(1 + o(1))n}$ of the $e^{(1+ o(1))n}$ residue classes.
\end{remark}

\section{Concluding remarks}
\label{sec:discussion}

In this paper, we have studied the Clifton--Huang weighting of hyperplanes, characterising the hyperplanes of maximum weight, and proving some stability, showing that all other hyperplanes have weight bounded away from $1$. We have then shown how this result can be used to determine further values of the extremal function $f(n,k)$, which is the minimum number of hyperplanes needed to cover all nonzero points of $\Qn$ at least $k$ times, while avoiding the origin completely. Several open problems remain, and we shall highlight some possible directions for further research below.

\paragraph{Maximal hyperplanes} We have shown that the weight of any hyperplane is either $1$ or at most $1 - \frac{1}{n}$. This latter bound is best possible, as evidenced, for example, by the plane $\Pi_c$ defined by $x_1 + x_2 + \hdots + x_{n-1} + cx_n = 1$, where $c \notin \{-(n-1), -(n-2), \hdots, 1\}$. However, as explained in Section~\ref{sec:covers}, when we are setting up the integer linear program for the covering problem, we need only consider hyperplanes that intersect $\{0,1\}^n$ maximally. Since $\Pi_1$ covers all the points that $\Pi_c$ does, and more, the above hyperplane is not maximal.

We could hope for stronger applications, then, if we could show that the weight of a \emph{maximal} hyperplane is either $1$ or much smaller. Unfortunately, there is not much we can gain here --- we can find maximal hyperplanes of weight $1 - \frac{1}{n-1}$. 
For example, it is known that the following hyperplanes, together with permutations of coefficients, are maximal hyperplanes of weight $1 - \frac{1}{n-1}$:
\begin{itemize}
    \item $x_1 + \cdots + x_{n-2} - k x_{n-1} - (n-2-k) x_n = 1$, $k = 1, \dots, \lfloor \frac{n-2}{2} \rfloor$.
    \item $x_1 + \cdots + x_{n-2} + 2 x_{n-1} - k x_n = 1$, $k = 1, \dots, n-3$.
    \item $x_1 + \cdots + x_{n-3} + 2 x_{n-2} - k x_{n-1} - (n-2-k) x_n = 1$, $k = 1, \dots, \lfloor \frac{n-2}{2} \rfloor$.
\end{itemize}
Still, it would be interesting to prove this larger separation for maximal hyperplanes, which could be significant for smaller values of $n$.

\begin{ques}
    Let $h$ be a hyperplane that intersects $\Qn$ maximally. Is it true that either $w(h) = 1$ or $w(h) \le 1 - \frac{1}{n-1}$?
\end{ques}

In our proof, with somewhat more involved arguments, we are able to establish the improved bound in some of our cases, including when we have a coefficient larger than $1$, or when we have a positive fractional coefficient. However, it remains to resolve the problem in the other cases.

Another direction that could help with reducing the search space for a wider range of parameters would be to characterise those hyperplanes whose weight is close to $1$. Initial steps in that direction were taken in Theorem~\ref{thm:furtherstab}, and it would be great to complete the picture.

\begin{ques}
    Can we characterise all hyperplanes $h$ with $w(h) = 1 - O\left(\frac{1}{n}\right)$?
\end{ques}

\paragraph{Improved lower bounds} In this setting, our lower bounds come from the fractional relaxation of the integer linear program, whose value was proven to be $H_n k$, where $H_n$ is the $n$th Harmonic number, by Clifton and Huang~\cite{clifton2020almost}. This immediately gives $f(n,k) \ge \ceil*{H_n k}$, and Clifton and Huang further proved that this lower bound is asymptotically tight.

However, it is not always sharp. As we have shown in Theorem~\ref{thm:plusone}, for every sufficiently large $n$ there are infinitely many choices of $k$ for which we have $f(n,k) \ge \ceil*{H_n k} + 1$. However, our methods do not allow us to prove any larger separation between the fractional and integer solutions, since an almost $k$-cover of size $\ceil*{H_n k} + 2$ could contain a hyperplane of very small weight. That said, in all the computational results we have obtained, we always have $f(n,k) \le \ceil*{H_n k} + 1$. It would be interesting to see how large $f(n,k) - \ceil*{H_n k}$ can be, and to develop methods for proving stronger lower bounds.

\begin{ques}
    Given $n$ and sufficiently large $k$, can we have $f(n,k) \ge \ceil{ H_n k } + 2$?
\end{ques}

It is worth reiterating a question of Clifton and Huang, who asked if the difference between the integer and fractional problems was bounded by an absolute constant.

\begin{ques}[Clifton--Huang~\cite{clifton2020almost}]
Is there an absolute constant $C > 0$, such that for every $n$ there are only finitely many $k$ with $f(n,k) \ge \ceil*{H_n k} + C$?
\end{ques}

\paragraph{Large dimensions} Finally, we note that while this linear programming approach is very fruitful in the cases we have considered, where $n$ is fixed and $k$ is large, Clifton and Huang showed the problem behaves very differently when $k$ is fixed and $n$ is large. The polynomial method has been fruitful in establishing lower bounds in this range, but Sauermann and Wigderson~\cite{SW22} showed that the solution to the algebraic problem is smaller than Clifton and Huang's conjectured value of $f(n,k)$.

\begin{conj}[Clifton--Huang~\cite{clifton2020almost}]
For $k \ge 2$ and $n$ sufficiently large, we have $f(n,k) = n + \binom{k}{2}$.
\end{conj}

\section*{Acknowledgement}
Shagnik Das is supported by Taiwan NSTC grant~111-2115-M-002-009-MY2. Wei-Hsuan Yu is supported by MOST under Grant No.~109-2628-M-008-002-MY4. Valjakas Djaljapayan and Yen-chi Roger Lin are partially supported by MOST under Grant No.~111-2115-M-003-009.

\bibliographystyle{plain}
\bibliography{almost-k}
\newpage

\appendix

\section{An almost $20$-cover of $Q^6$} \label{app:20cover}
We used an integer linear program solver to find an almost $20$-cover of $Q^6$ that consists of 49 hyperplanes.
Each row of Table~\ref{tbl:f62049} lists the coefficients $c_i$ that give the hyperplane of the form $\sum_{i=1}^6 c_i x_i = 1$.

\begin{table}[h!]
  \centering
  \caption{The coefficients for hyperplanes $\vec{c}\cdot\vec{x}=1$ in an almost 20-cover in $Q_*^6$, where $\vec{c} = (c_1, c_2, c_3, c_4, c_5, c_6)$}
  \label{tbl:f62049}
  \begin{tabular}[hb]{rrrrrr}
  $c_1$ & $c_2$ & $c_3$ & $c_4$ & $c_5$ & $c_6$ \\ \hline
1 & 1 & 1 & 0 & $-1$ & $-1$ \\
1 & 1 & 1 & 0 & $-2$ & 0 \\
1 & 1 & 1 & $-1$ & $-1$ & 0 \\
1 & 1 & 0 & 1 & $-1$ & 0 \\
1 & 1 & 0 & $-2$ & 1 & 1 \\
1 & 1 & $-1$ & 1 & 1 & $-1$ \\
1 & 1 & $-1$ & 1 & 0 & $-1$ \\
1 & 1 & $-1$ & 0 & 0 & 1 \\
1 & 1 & $-2$ & 1 & 0 & 0 \\
1 & 0 & 1 & 0 & 1 & $-1$ \\
1 & 0 & 1 & 0 & 0 & 0 \\
1 & 0 & 1 & 0 & 0 & $-1$ \\
1 & 0 & 1 & $-1$ & $-1$ & 1 \\
1 & 0 & 0 & 1 & 1 & 1 \\
1 & 0 & 0 & 1 & $-1$ & 0 \\
1 & 0 & 0 & 0 & 1 & 0 \\
1 & 0 & 0 & 0 & 0 & 1 \\
1 & 0 & $-1$ & 1 & 1 & 0 \\
1 & 0 & $-1$ & 1 & 0 & 1 \\
1 & $-2$ & 0 & 0 & 1 & 1 \\
0 & 1 & 1 & 1 & 1 & 0 \\
0 & 1 & 1 & 1 & 0 & $-1$ \\
0 & 1 & 1 & 0 & 0 & 1 \\
0 & 1 & 1 & 0 & 0 & 0 \\
0 & 1 & 0 & 1 & $-2$ & 1
  \end{tabular} 
  \hspace{3ex}
    \begin{tabular}[ht]{rrrrrr}
  $c_1$ & $c_2$ & $c_3$ & $c_4$ & $c_5$ & $c_6$ \\ \hline
0 & 1 & 0 & 0 & 1 & 0 \\
0 & 1 & 0 & 0 & 1 & 0 \\
0 & 1 & 0 & 0 & 0 & 1 \\
0 & 1 & 0 & 0 & 0 & 0 \\
0 & 1 & $-1$ & 1 & 1 & 1 \\
0 & 0 & 1 & 1 & 1 & $-1$ \\
0 & 0 & 1 & 1 & 0 & 0 \\
0 & 0 & 1 & 1 & 0 & 0 \\
0 & 0 & 1 & 0 & 1 & 0 \\
0 & 0 & 1 & 0 & 0 & 1 \\
0 & 0 & 1 & 0 & 0 & 1 \\
0 & 0 & 1 & 0 & 0 & 0 \\
0 & 0 & 1 & $-1$ & 1 & 0 \\
0 & 0 & 0 & 1 & 1 & 0 \\
0 & 0 & 0 & 1 & 0 & 1 \\
0 & 0 & 0 & 1 & 0 & 0 \\
0 & 0 & 0 & 0 & 1 & 1 \\
0 & 0 & 0 & 0 & 1 & 0 \\
0 & 0 & 0 & 0 & 1 & 0 \\
0 & 0 & 0 & 0 & 0 & 1 \\
0 & 0 & $-1$ & 1 & 0 & 1 \\
0 & 0 & $-1$ & 0 & 1 & 1 \\
0 & $-1$ & 1 & 0 & 1 & 1 \\
$-1$ & 1 & 0 & 1 & $-1$ & 1
  \end{tabular}
\end{table}
\end{document}